\documentclass[a4paper,12pt,final]{amsart}
\usepackage{times,a4wide,mathrsfs,amssymb,amsmath,amsthm,wasysym}

\newcommand{\C}{\mathbb{C}}
\newcommand{\ZZ}{\mathbb{Z}}

\newcommand{\QQ}{\mathbb{Q}}
\newcommand{\NN}{\mathbb{N}}
\newcommand{\PP}{\mathbb{P}}

\newcommand{\OO}{\mathcal O}
\newcommand{\Ss}{\mathcal S}

\newcommand{\DD}{\mathcal D}

\newcommand{\VV}{\mathcal V}

\newcommand{\MM}{\mathcal M}

\newcommand{\wt}{\widetilde}
\newcommand{\rom}{\romannumeral}

 \makeatletter
\newcommand*{\da@rightarrow}{\mathchar"0\hexnumber@\symAMSa 4B }
\newcommand*{\da@leftarrow}{\mathchar"0\hexnumber@\symAMSa 4C }
\newcommand*{\xdashrightarrow}[2][]{%
  \mathrel{%
    \mathpalette{\da@xarrow{#1}{#2}{}\da@rightarrow{\,}{}}{}%
  }%
}
\newcommand{\xdashleftarrow}[2][]{%
  \mathrel{%
    \mathpalette{\da@xarrow{#1}{#2}\da@leftarrow{}{}{\,}}{}%
  }%
}
\newcommand*{\da@xarrow}[7]{%
  \sbox0{$\ifx#7\scriptstyle\scriptscriptstyle\else\scriptstyle\fi#5#1#6\m@th$}%
  \sbox2{$\ifx#7\scriptstyle\scriptscriptstyle\else\scriptstyle\fi#5#2#6\m@th$}%
  \sbox4{$#7\dabar@\m@th$}%
  \dimen@=\wd0 %
  \ifdim\wd2 >\dimen@
    \dimen@=\wd2 %
  \fi
  \count@=2 %
  \def\da@bars{\dabar@\dabar@}%
  \@whiledim\count@\wd4<\dimen@\do{%
    \advance\count@\@ne
    \expandafter\def\expandafter\da@bars\expandafter{%
      \da@bars
      \dabar@ 
    }%
  }%
  \mathrel{#3}%
  \mathrel{%
    \mathop{\da@bars}\limits
    \ifx\\#1\\%
    \else
      _{\copy0}%
    \fi
    \ifx\\#2\\%
    \else
      ^{\copy2}%
    \fi
  }%
  \mathrel{#4}%
}
\makeatother

\DeclareMathOperator{\aut}{Aut}

\DeclareMathOperator{\ide}{id}

\DeclareMathOperator{\ord}{ord}

\newtheorem{theorem}{Theorem}[section]

\newtheorem{lemma}[theorem]{Lemma}
\newtheorem{sublemma}[theorem]{Sublemma}
\newtheorem{corollary}[theorem]{Corollary}
\newtheorem{proposition}[theorem]{Proposition}
\newtheorem{conjecture}[theorem]{Conjecture}
\newtheorem{nonumbering}{Theorem}

\newtheorem{convention}{Conventions}

\theoremstyle{definition}
\newtheorem{remark}[theorem]{Remark}
\newtheorem{definition}[theorem]{Definition}

\newtheorem{notation}[theorem]{Notation}

\newtheorem{nonumberingt}{Acknowledgements}

\begin{document}

\author[Robert Laterveer]
{Robert Laterveer}

\address{Institut de Recherche Math\'ematique Avanc\'ee,
CNRS -- Universit\'e 
de Strasbourg,\
7 Rue Ren\'e Des\-car\-tes, 67084 Strasbourg CEDEX,\
FRANCE.}
\email{robert.laterveer@math.unistra.fr}

\title{Bloch's conjecture for some numerical Campedelli surfaces}

\begin{abstract} We prove Bloch's conjecture for numerical Campedelli surfaces with fundamental group of order $9$. 
\end{abstract}


\maketitle


A {\em numerical Campedelli surface\/} is a minimal surface $S$ of general type with $p_g(S)=0$ (and hence $q(S)=0$) and $K_S^2=2$.
Examples of such surfaces are given in \cite{Cam}, \cite{Pet}, \cite{R1}, \cite{R2}, \cite{Sup}, \cite{Kul}, \cite{LP}.
It is known that the order of the algebraic fundamental group $G=\pi_1^{alg}(S)$ of a numerical Campedelli surface ranges from $0$ to $9$ \cite{R1}, \cite[Chapter VII.10]{BPHV}. Numerical Campedelli surfaces with $G=\ZZ_2^3$ are called {\em classical Campedelli surfaces\/} \cite{Cam}, \cite{Kul}.

For any smooth projective variety $M$, let $A^i(M):=CH^i(M)$ denote the Chow groups (i.e. codimension $i$ algebraic cycles modulo rational equivalence).
It is known since Mumford's work \cite{Mum} that if $S$ is a surface with trivial Chow group of $0$-cycles (i.e. $A^2(S)_{\QQ}=\QQ$), then
$p_g(S)=q(S)=0$. Bloch has famously conjectured the converse:

\begin{conjecture}[Bloch \cite{B}]\label{blochconj} Let $S$ be a surface with $p_g(S)=q(S)=0$. Then
  \[ A^2(S)=\ZZ\ .\]
  \end{conjecture}
  
  Bloch's conjecture has been hugely influential in the development of the subject of algebraic cycles (cf. \cite{Jan}, \cite{MNP}, \cite{Vo} for fascinating overviews of the field). It is notoriously still open for surfaces of general type. 
  
 Conjecture \ref{blochconj} is known to be true for classical Campedelli surfaces (this is a nice argument based on the presence of involutions, cf. \cite[Theorem 3]{IM} or alternatively \cite[Theorem 5.1]{PW}). Voisin \cite{V8} has proven conjecture \ref{blochconj} for a family of numerical Campedelli surfaces with $\ord (G)=5$. 
The aim of this note is to prove Bloch's conjecture for numerical Campedelli surfaces with maximal fundamental group:

\begin{nonumbering}[=theorem \ref{main}] Let $S$ be a numerical Campedelli surface with $\ord(G)=9$. Then
  \[ A^2(S)=\ZZ\ .\]
  \end{nonumbering}
  
 The surfaces covered by theorem \ref{main} correspond to two irreducible components of the moduli space (of surfaces of general type with $p_g=0$ and $K^2=2$), one of dimension $6$ and one of dimension $7$ \cite[Theorem 4.2]{MLP}.
  
As a corollary, we also obtain Bloch's conjecture for certain numerical Godeaux surfaces (corollary \ref{cor}). 

The argument proving theorem \ref{main} follows the pattern of Voisin's seminal method of {\em spread\/} of algebraic cycles in families \cite{V0}, \cite{V1}, \cite{Vo}, \cite{V8}. In the case at hand, this method can be applied thanks to work of Mendes Lopes--Pardini, who furnish a nice explicit construction of numerical Campedelli surfaces with $G$ of order $9$ (\cite{MLP}, cf. also theorem \ref{const} below).
A central point of the argument is the vanishing of a certain Griffiths group of the fibre product of the family (proposition \ref{van}); this is proven using a stratification argument and some basic properties of Lawson homology.

 \vskip0.6cm

\begin{convention} In this article, the word {\sl variety\/} will refer to a reduced irreducible scheme of finite type over $\C$. A {\sl subvariety\/} is a (possibly reducible) reduced subscheme which is equidimensional. 

We will denote by $A_j(X)$ the Chow group of $j$-dimensional cycles on $X$; for $X$ smooth of dimension $n$ the notations $A_j(X)$ and $A^{n-j}(X)$ are used interchangeably. We will write $B_j(X)$ for the group of $j$-dimensional cycles modulo algebraic equivalence, and $B^{n-j}(X)=B_j(X)$ when $X$ is smooth of dimension $n$.

The notations $A^j_{hom}(X)$, $A^j_{AJ}(X)$ will be used to indicate the subgroups of homologically trivial, resp. Abel--Jacobi trivial cycles.
Likewise, we will write $B^i_{hom}(X)$ for the subgroup of homologically trivial cycles (the groups $B^i_{hom}(X)$ are traditionally known as the {\em Griffiths groups\/} of $X$).


We use $H^j(X)$ to indicate singular cohomology $H^j(X,\QQ)$, and $H_j(X)$ to indicate Borel--Moore homology $H_j^{BM}(X,\QQ)$.
\end{convention}

\section{Preliminaries}

\subsection{Singular surfaces}

This subsection gathers some properties of singular surfaces that are Alexander schemes, in the sense of Vistoli and Kimura. These surfaces $S$ have the desirable property that the Chow groups with rational coefficients of $S$ and all its powers $S^m$ behave just as well as Chow groups of smooth varieties.

\begin{proposition}\label{alex} Let $S$ be a normal projective surface, and assume there is a resolution of singularities such that the exceptional divisor is a union of rational curves. Then $S^m$ is an Alexander scheme in the sense of \cite{Vis}, \cite{Kim0}. In particular, the formalism of correspondences with rational coefficients works for $S^m$ just as for smooth projective varieties.
\end{proposition}

\begin{proof} The fact that $S$ is an Alexander scheme is a result of Vistoli's:

\begin{theorem}[Vistoli \cite{Vis}]\label{vis} Let $S$ be a normal surface. The following are equivalent:

\noindent
(\rom1) $S$ is an Alexander scheme;

\noindent
(\rom2) there exists a resolution of singularities of $S$ such that  the exceptional divisor is a union of rational curves.
\end{theorem}

\begin{proof} This is 
\cite[Theorem 4.1]{Vis}.
\end{proof}

 Since the property ``being an Alexander scheme'' is stable under products \cite[Remark 2.7(\rom1)]{Kim0}, $S^m$ is an Alexander scheme. (NB: the fact that the definitions of Alexander scheme given in \cite{Vis} and \cite{Kim0} coincide for connected schemes is 
\cite[Corollary 4.5]{Kim0}.)

The fact that products of Alexander schemes are Alexander, and that Chow groups of an Alexander scheme have an intersection product and a pullback, means that the formalism of correspondences can be extended from smooth projective varieties to projective Alexander schemes.
\end{proof}

We will use proposition \ref{alex} in the following guise:

\begin{proposition}\label{gen} Let $\VV\to B$ be a flat projective morphism of relative dimension $2$, where $B\subset\PP^r$ is a Zariski open, and assume there is a finite group $G$ acting on $\VV$ and preserving the fibres $V_b$. Assume moreover the following:

\noindent
(\rom1) any fibre $V_b$ is a projective surface that is Alexander, and $H^1(V_b)=0$;

\noindent
(\rom2) for very general $b\in B$, the fibre $V_b$ is a quotient variety (i.e. a global quotient $V_b=V^\prime_b/H$ where $V^\prime_b$ is smooth and $H$ is a finite group), and one has $A_0^{hom}(V_b)_{\QQ}^G=0$.
 
Then 
  \[ A_0^{hom}(V_b)_{\QQ}^G=0\ \ \  \forall \ b\in B\ .\]
\end{proposition}

\begin{proof} Taking a very general line in $B$ passing through a given $b$, we may assume without loss of generality that $r=\dim B=1$.

Let us consider the relative cycle
  \[ \Delta_G:=  {1\over G}\, {\displaystyle\sum_{g\in G}}\Gamma_g\ \ \ \in A_{3}(\VV\times_B \VV)_{\QQ}\ .\]
  For very general $b\in B$, the restriction 
  \[ \Delta_G\vert_b:= \Delta_G\vert_{V_b\times V_b}\ \ \ \in A_2(V_b\times V_b)_{\QQ} \]
  is well--defined (using the refined Gysin homomorphism \cite{F}), and its cohomology class is such that
  \[ \Delta_G\vert_b\ \ \ \in\ H^4(V_b\times V_b)\cong H_4(V_b\times V_b) \]
  (since $V_b\times V_b$ is a quotient variety).
  Let us consider the K\"unneth decomposition
   \[ H^4(V_b\times V_b)\cong H^4(V_b)\otimes H^0(V_b)\oplus H^2(V_b)\otimes H^2(V_b)\oplus H^0(V_b)\otimes H^4(V_b)\ .\]
   Taking $h\in A^1(\VV)$ a relatively ample Cartier divisor, one can define
   \[ \begin{split} \Delta_G^0 &:=  \alpha_0\,  {\displaystyle\sum_{g\in G}}  (p_1)^\ast g^\ast(h^2) \ \ \ \hbox{in}\ A_{3}(\VV\times_B \VV)_{\QQ}\ ,\\
                          \Delta_G^4 &:=  \alpha_4\,  {\displaystyle\sum_{g\in G}}  (p_2)^\ast g^\ast(h^2) \ \ \ \hbox{in}\ A_{3}(\VV\times_B \VV)_{\QQ}\ ,\\   
                         \end{split}\]
                 where $p_1, p_2\colon \VV\times_B \VV\to \VV$ denote projection on the first, resp. second summand.
           For appropriate constants $\alpha_j\in\QQ$, and for very general $b\in B$, these relative cycles will be such that the restrictions
         \[           \Delta_G^0\vert_b\ ,\ \ \ \Delta_G^4\vert_b\ \ \ \in H^4(V_b\times V_b)   \]
      are the first, resp. the last components of the K\"unneth decomposition of the restriction $\Delta_G\vert_b$. That is, the relative cycle 
      \[ \Delta_G^- := \Delta_G - \Delta_G^0 - \Delta_G^4\ \ \ \in\   A_{3}(\VV\times_B \VV)_{\QQ}\ \]
    has the property that for very general $b\in B$, the cohomology class of the restriction satisfies
     \[ \Delta_G^-\vert_b\ \ \ \in H^2(V_b)\otimes H^2(V_b)\ \ \ \subset\ H^4(V_b\times V_b)\ .\]
     
     By assumption, the Chow group $A_0(V_b)_{\QQ}^G$ is trivial for very general $b\in B$. The Bloch--Srinivas argument \cite{BS} still goes through for quotient varieties, and so this implies $H^{2,0}(V_b)^G=0$. The Lefschetz (1,1) theorem (which still goes through for quotient varieties) thus implies that there is an isomorphism
     \[ A^1(V_b)_{\QQ}^G\ \xrightarrow{\cong}\ H^2(V_b)^G \]
     (here $A^1(V_b)$ denotes the Picard group, which with rational coefficients is the same as the group of Weil divisors $A_1(V_b)$ since $V_b$ is $\QQ$-factorial).
      
     The upshot of all this is as follows: we have a relative cycle $\Delta_G^-$, and for very general $b\in B$, there exist a divisor $D_b\subset V_b$ and a cycle $\gamma_b\in A^2(V_b\times V_b)_{\QQ}$ supported on $D_b\times D_b$, such that
     \[ \Delta_G^-\vert_b=\gamma_b\ \ \ \hbox{in}\ H^4(V_b\times V_b)\ .\]
    Applying Voisin's Hilbert schemes argument \cite[Proposition 3.7]{V0} to this set--up, we can find a divisor $\DD\subset\VV$, and a relative cycle $\gamma\in A_{}(\VV\times_B \VV)_{\QQ}$ supported on $\DD\times_B \DD$, with the property that 
      \[ (\Delta_G^- -\gamma)\vert_b =0\ \ \ \hbox{in}\ H^4(V_b\times V_b)\ ,\ \ \ \hbox{for\ very\ general\ }b\in B\ .\]             
       
   We now make the following observation:  the assumption that the Chow group $A_0(V_b)_{\QQ}^G$ is trivial (for very general $b\in B$) means that the motive 
   \[ (V_b,\Delta_G,0):=(V_b^\prime, \Delta_{G\times H},0)=(V_b^\prime,\sum_{g\in G\times H} \Gamma_g,0)\ \ \ \in\ \MM_{\rm rat} \]
    is finite-dimensional, in the sense of \cite{Kim}.  As such, the homologically trivial correspondence
    \[      (\Delta_G^- -\gamma)\vert_b \ \ \ \hbox{in}\ A^2_{hom}(V_b\times V_b)_{\QQ}  \]
    is nilpotent, for very general $b\in B$. That is, for a very general $b\in B$ there exists $N_b\in\NN$ such that
    \[   \Bigl( (\Delta_G^- -\gamma)\vert_b\Bigr){}^{\circ N_b} =0\ \ \ \hbox{in}\ A^2_{}(V_b\times V_b)_{\QQ} \ . \]
    What's more, as the nilpotence index $N_b$ is bounded in terms of the Betti numbers of $V_b$, there actually exists one global integer $N\in\NN$ such that
      \[   \Bigl( (\Delta_G^- -\gamma)\vert_b\Bigr){}^{\circ N} =0\ \ \ \hbox{in}\ A^2_{}(V_b\times V_b)_{\QQ}\ ,\ \ \  \hbox{for\ very\ general\ }b\in B\ .\]     
      Developing this expression, one obtains
      \[ ( \Delta_G^{-}\vert_b)^{\circ N} -\gamma^\prime_b=0\ \ \ \hbox{in}\ A^2_{}(V_b\times V_b)_{\QQ}\ ,\ \ \  \hbox{for\ very\ general\ }b\in B\ ,\]
      where the cycle $\gamma^\prime_b$ is a correspondence supported on $D_b\times D_b\subset V_b\times V_b$, and $D_b:=\DD\cap V_b$ is a divisor.
      But the restriction $\Delta_G^{-}\vert_b$ is idempotent (this is easily checked directly), and so
       \[  \Delta_G^{-}\vert_b -\gamma^\prime_b=0\ \ \ \hbox{in}\ A^2_{}(V_b\times V_b)_{\QQ}\ ,\ \ \  \hbox{for\ very\ general\ }b\in B\ .\]
             Applying Voisin's \cite[Proposition 3.7]{V0} once more, one can find a relative cycle $\gamma^{\prime\prime}\in A_3(\VV\times_B \VV)_{\QQ}$ supported on $\DD\times_B \DD$ with the property that
             \[  \Bigl(\Delta_G^{-} -\gamma^{\prime\prime}\Bigr)\vert_b=0\ \ \ \hbox{in}\ A^2_{}(V_b\times V_b)_{\QQ}\ ,\ \ \  \hbox{for\ very\ general\ }b\in B\ .\]
     At this point, we invoke the following result:      
       
       \begin{lemma}\label{voi} Let $M\to B$ be a projective morphism between quasi--projective varieties, and assume $B$ is smooth of dimension $d$. Let $\Gamma\in A_j(M)$. Then the set of points $b\in B$ such that the restriction $\Gamma\vert_{M_b}$ is zero in $A_{j-d}(M_b)$ is a countable union of closed algebraic subsets of $B$.
  \end{lemma}
  
    \begin{proof} This is \cite[Proposition 2.4]{V10}.
   (This result is also stated as \cite[Lemma 3.2]{Vo}, where it is assumed that $M$ is smooth.) 
    
  \end{proof} 

  Applying lemma \ref{voi} to $\Gamma=  \Delta_G^{-} -\gamma^{\prime\prime}$, we find that  
       \[  \Bigl(\Delta_G^{-} -\gamma^{\prime\prime}\Bigr)\vert_b=0\ \ \ \hbox{in}\ A^2_{}(V_b\times V_b)_{\QQ}\ ,\ \ \  \hbox{for\ all\ }b\in B\ .\]
       
    But for any $b\in B$, the restriction $\Delta_G^-\vert_b$ is a projector on $A_0^{hom}(V_b)^G_{\QQ}$. As for the restriction $\gamma^{\prime\prime}\vert_b$, by construction this is a correspondence supported on $D_b\times D_b\subset V_b\times V_b$, where $D_b:=\DD\cap V_b$. Since all irreducible components of $\DD$ are $2$-dimensional varieties dominating the base $B$, $D_b\subset V_b$ is a divisor for any $b\in B$. Because $V_b$ is (by assumption) an Alexander scheme, correspondences act on $A_\ast(V_b)_{\QQ}$. The action of $\gamma^{\prime\prime}\vert_b$ on $A_0^{hom}(V_b)_{\QQ}=A_0^{AJ}(V_b)_{\QQ}$ factors over $A_1^{AJ}(D_b)_{\QQ}=0$. This proves that
    \[ A_0^{hom}(V_b)^G_{\QQ}=0\ \ \ \hbox{for\ all\ }b\in B\ .\]  
       
        \end{proof}

\begin{remark} The proof of proposition \ref{gen} becomes easier when $\VV$ is smooth (in that case, one can use the theory of relative correspondences, as in \cite[Chapter 8]{MNP}).

A result related to proposition \ref{gen} is \cite[Theorem 3.1]{KSZ}. The result of \cite{KSZ} is stronger than proposition \ref{gen}, as there is no condition on the singularities of the fibres.
\end{remark}

For later use, we also note the following:

\begin{lemma}\label{alexquot} Let $V$ be a normal surface that is an Alexander scheme, and let $G\subset\aut(V)$ be a finite abelian group. The quotient $S:=V/G$ is an Alexander scheme, and
  \[ A_i(S)_{\QQ}\cong A_i(V)_{\QQ}^G\ .\]
\end{lemma}

\begin{proof} (To be sure, this is well--known for smooth varieties $V$.) 

    The first statement follows from \cite[Proposition 2.11]{Vis}.
    
    Let $p\colon V\to S$ be the quotient morphism. Since $S$ is Alexander, there is a well-defined pullback
      \[ p^\ast\colon\ \ A_i(S)\ \to\ A_i(V)\ ,\]
    and there is equality
    \[ p^\ast p_\ast = \sum_{g\in G}\, g_\ast\colon\ \ \ A_i(V)\ \to\ A_i(V)\ .\]
    (To prove this equality, one can use the good functorial properties of Alexander schemes to pull back to a resolution of singularities $\wt{V}$, where the equality is well--known.) In particular,
    \[ p^\ast p_\ast = \ord(G)\cdot\ide\colon\ \ \ A_i(V)^G\ \to\ A_i(V)^G\ ,\]
    which proves the lemma.
      \end{proof}

 \begin{proposition}\label{birat} Let $S_1, S_2$ be normal projective surfaces that are Alexander, and assume $S_1$ and $S_2$ are birational. Then
   \[ A_0(S_1)_{\QQ}\cong A_0(S_2)_{\QQ}\ .\]
   \end{proposition}
   
   \begin{proof} This is probably well-known. The point is that there are isomorphisms
     \[ A^2(S_j)_{\QQ}\ \xrightarrow{\cong}\ A_0(S_j)_{\QQ}\ \ \ (j=1,2)\ ,\]
     where $A^2()$ denotes operational Chow cohomology \cite{Kim-1}. Let $\wt{S_j}\to S_j$ denote resolutions of singularities. Then the exact sequences of \cite{Kim-1} imply that there are isomorphisms 
     \[ A^2(S_j)_{\QQ}\ \xrightarrow{\cong}\ A^2(\wt{S_j})_{\QQ}\ \ \ (j=1,2)\ .\]
     But the smooth surfaces $\wt{S_1}$ and $\wt{S_2}$ are birational, and so
     \[   A^2(\wt{S_1})_{\QQ}\ \xrightarrow{\cong}\ A^2(\wt{S_2})_{\QQ}  \ .\] 
   \end{proof}

\subsection{A weak nilpotence result}

A particular case of Kimura's finite-dimensionality conjecture \cite{Kim} is the conjecture that a numerically trivial self-correspondence of degree $0$ is nilpotent. The following inconditional result establishes a weak version of this conjecture:

\begin{proposition}[Voevodsky \cite{Voe}, Voisin \cite{V9}]\label{weaknilp} Let $X$ be a smooth projective variety of dimension $n$. Let $\Gamma\in A^n_{alg}(X\times X)$. Then $\Gamma$ is nilpotent.
\end{proposition}

\begin{proof} This is \cite{Voe} or \cite[2.3.1]{V9} or \cite[Theorem 3.25]{Vo}.

\end{proof}

\subsection{Lawson homology}

\begin{theorem}[Friedlander, Friedlander--Mazur \cite{Fr}, \cite{FM}]\label{laws} There exists a bigraded homology theory $L_\ast H_\ast()$ of $\QQ$--vector spaces, called {\em Lawson homology\/}, with the following properties:

\noindent
(\rom1) the assignment $L_\ast H_\ast()$ is part of a Poincar\'e duality theory (in the sense of \cite{BO});

\noindent
(\rom2) if $P\to M$ is a projective bundle with fibres $\PP^r$, there is a functorial isomorphism
       \[  L_j H_i(P)\cong \bigoplus_{k=0}^{r} L_{j-k} H_{i-2k}(M)\ ;\]
       
    \noindent
    (\rom3) there exist functorial maps
     \[ L_j H_i()\ \to\ L_{j-1} H_i()\ ;\]
     
   \noindent
   (\rom4) for any variety $M$, we have
        \[   L_j H_i(M) = \begin{cases}   B_j(M)_{\QQ} & \hbox{if}\ i=2j\ ;\\
                                   H_i(M) & \hbox{if}\ j=0\ .\\
                                  \end{cases}      \]  

\end{theorem}

\begin{proof} Point (\rom1) is contained in \cite{Fr}. Point (\rom2) is \cite[Proposition 3.1]{Hu0}. The last two points are in \cite{FM}.

\end{proof}

\subsection{Weak and strong property}

\begin{definition}\label{ws} Let $X$ be a quasi--projective variety, and $m\in\NN$. We say that $X$ has the {\em weak property in degree $m$\/} if the natural maps
  \[ L_j H_{2j}(X)\ \to\ H_{2j}(X) \]
  are injective for $j\ge m$ (i.e., $B_j^{hom}(X)=0$ for $j\ge m$).
  
  We say that $X$ has the {\em strong property in degree $m$\/} if $X$ has the weak property in degree $m$, and in addition the natural maps
  \[ L_j H_{2j+1}(X)\ \to\ H_{2j+1}(X) \]
  are surjective for $j\ge m$.
\end{definition}

\begin{lemma}\label{paste} Let $X$ be a quasi--projective variety, and $Y\subset X$ a closed subvariety with complement $U:=X\setminus Y$. Assume that $U$ has the strong property in degree $m$, and $Y$ has the weak property in degree $m$. Then $X$ has the weak property in degree $m$.
\end{lemma}

\begin{proof} The good functorial properties of Lawson homology ensure the existence of a commutative diagram with exact rows
  \[\begin{array}[c]{cccccccccc}
     L_j H_{2j+1}(X) &\to&   L_j H_{2j+1}(U) &\to& L_j H_{2j}(Y) &\to& L_j H_{2j}(X)&\to& L_j H_{2j}(U) &\to 0\\
    \ \ \downarrow{\scriptstyle \nu_0} && \ \  \downarrow{\scriptstyle \nu_1} &&\ \ \downarrow{\scriptstyle \nu_2}&&\ \ \downarrow{\scriptstyle \nu_3}&&\ \ \downarrow{\scriptstyle \nu_4}&\\
       H_{2j+1}(X) &\to&      H_{2j+1}(U) &\to&  H_{2j}(Y) &\to&  H_{2j}(X)&\to&  H_{2j}(U) &\to \\
        \end{array}\]

Consider now a degree $j\ge m$. By assumption, the arrow labelled $\nu_1$ is surjective, and the arrows labelled $\nu_2, \nu_4$ are injective. An easy diagram chase then shows that the arrow labelled $\nu_3$ is injective, i.e. $X$ has the weak property in degree $m$.

\end{proof}

\begin{lemma}\label{cut}  Let $X$ be a quasi--projective variety, and $Y\subset X$ a closed subvariety of dimension $d$ with complement $U:=X\setminus Y$. Assume that $X$ has the strong property in degree $d$. Then $U$ has the strong property in degree $d$.

\end{lemma}

\begin{proof} This is proven by doing another diagram chase in the commutative diagram of the proof of lemma \ref{paste}, the crux being that the arrow $\nu_2$ is an isomorphism for $j\ge d$.

\end{proof}

 \begin{lemma}\label{projbun} Let $P\to M$ be a projective bundle, with fibres isomorphic to $\PP^r$. Assume that $M$ has the weak (resp. strong) property in degree $m$. Then $P$ has the weak (resp. strong) property in degree $m+r$.
  \end{lemma} 
  
  \begin{proof} Immediate from theorem \ref{laws}(\rom2).
  
   \end{proof}

\begin{remark} Definition \ref{ws} and lemmas \ref{paste} and \ref{cut} are directly inspired by work of Totaro \cite{Tot}, where a similar notion is used (with higher Chow groups instead of Lawson homology) to prove that linear varieties have trivial Chow groups. Subsequently, in \cite{moit} I used Totaro's notion to prove triviality of a certain Chow group.

\end{remark}
 
\section{Construction of the surfaces}

We rely on the following explicit construction of Campedelli surfaces with maximal torsion of the algebraic fundamental group $G=\pi_1^{alg}(S)$:

\begin{theorem}[Mendes Lopes--Pardini \cite{MLP}]\label{const} Let $S$ be a numerical Campedelli surface with $\ord (G)=9$. The canonical model $\bar{S}$ of $S$ is one of the following:

\noindent
(\rom1) (``Type A'':)  $\bar{S}=V/G$, where $G=\ZZ_9$ and $V\subset W$ is a $G$-invariant divisor in the linear system $\vert 3H\vert$ on $W:=\PP^1\times\PP^1\times\PP^1$. Here $H$ is ${\mathcal O}_W(1,1,1)$ and $G$ acts on $W$ by
  \[   g(x,y,z):= (y,z,\nu x) \ ,\]
  where $x:=x_1/x_0, y:=y_1/y_0, z:=z_1/z_0$ are affine coordinates, and $\nu$ is a primitive $3$rd root of unity. The divisor $V$ is a surface 
  with at most rational double points, not containing the fixed points of $g^3$. Conversely, any such $V/G$ is the canonical model of a numerical Campedelli surface with $G=\ZZ_9$.
    
  \noindent
  (\rom2) (``Type B1'':) $\bar{S}=V/G$, where $G=\ZZ_3^2$ and $V\subset W$ is a $G$-invariant divisor in the linear system $\vert 3H\vert$ on the flag variety
    \[ W:=\{ x_0y_0 +x_1 y_1 + x_2 y_2=0 \}\ \ \ \subset\ \PP^2\times (\PP^2)^\ast\ .\]
    Here $H:={\mathcal O}_{\PP^2\times\PP^2\ast}(1,1)\vert_W$, and the $G$--action on $W$ is induced by the $G$--action on $\PP^2$ given by
      \[ g_1(x_0,x_1,x_2):=(x_0,\nu x_1,\nu^2 x_2)\ ,\ \ \ g_2(x_0,x_1,x_2):=(x_1,x_2,x_0)\ .\]
      The divisor $V$ is a surface 
  with at most rational double points, not containing any point fixed by a non-trivial element of $G$. Conversely, any such $V/G$ is the canonical model of a numerical Campedelli surface with $G=\ZZ_3^2$.
  
  \noindent
  (\rom3) (``Type B2'':) $\bar{S}=V/G$, where $G=\ZZ_3^2$ and $V$ is a surface with rational double point singularities. Moreover, $\bar{S}$ is a member of a family of surfaces $\Ss\to B$, and the general member of the family is (isomorphic to) a surface $\bar{S}$ which is as in (\rom2) and which is a quotient surface. 
  \end{theorem}

\begin{proof} The three cases correspond to what is called ``type A'', resp. ``type B1'', resp. ``type B2'' in \cite{MLP}.
The fact that $\bar{S}$ is a surface of one of these three types is \cite[Theorem 3.1]{MLP}, combined with \cite[Proposition 3.12]{MLP} for the description of type B2 given in (\rom3). To see that the general element of the family in case (\rom3) is a quotient variety, one proceeds as follows: given $S$ such that $\bar{S}=V_0/G$ is of type B2, the construction of \cite[Proposition 3.12]{MLP} allows one to take a pencil spanned by $V_0$ and some smooth $V_1$, where $V_1$ is such that the quotient $V_1/G$ is of type B1. The general element of the pencil $\VV_b/G$ will be a quotient variety of type B1, as requested.

The fact that any $V/G$ as in (\rom1) or (\rom2) is the canonical model of a numerical Campedelli surface is \cite[Proposition 2.3]{MLP}.

%
  
\end{proof}

\begin{remark} The moduli dimension is $6$ in case (\rom1), $7$ in case (\rom2), and $6$ in case (\rom3) \cite[Section 3]{MLP}. Surfaces as in (\rom1) form an irreducible component of the moduli space. Surfaces as in (\rom2) and (\rom3) are in the same ($7$-dimensional) irreducible component of the moduli space \cite[Theorem 4.2]{MLP}.

As a corollary of the classification of theorem \ref{const}, for a numerical Campedelli surface $S$ with $\ord(G)=9$ one has
  \[ \pi_1^{alg}(S)=\pi_1^{top}(S)\ \]
  \cite[Proposition 4.3]{MLP}.

\end{remark}

\begin{notation}\label{not} We consider the following families of surfaces $\VV\to B$ and $\Ss\to B$:

\noindent
(\rom1) (``Type A'':) The family
  \[ \VV\ \subset\ B\times W \]
  is the universal family of all divisors $V_b$ in $W:=\PP^1\times \PP^1\times \PP^1$ as in theorem \ref{const}(\rom1), so each $V_b$ has at most rational double points and the generic $V_b$ is smooth. (Here $B\subset\PP H^0(W,\OO_W(3))^G$ is the Zariski open parametrizing divisors $V_b$ that are as in theorem \ref{const}(\rom1).)  The family
  \[ \VV^0\ \to\ B^0 \]
  is the universal family of all {\em smooth\/} divisors $V_b\subset W$ in the linear system $\vert 3H\vert$, so $\VV^0\subset\VV$ is a Zariski open.
  
  The family $\Ss$ is defined as
  \[ \Ss:= \VV/G\ \to\ B \]
  (where $G:=\ZZ_9$ and the $G$-action on $\VV$ is induced by the $G$-action on $W$).

 \noindent
  (\rom2) (``Type B1'':) The family
  \[ \VV\ \subset\ B\times W \]
  is the universal family of all divisors $V_b$ in the flag variety $W$ as in theorem \ref{const}(\rom2), so each $V_b$ has at most rational double points and the generic $V_b$ is smooth. The family
  \[ \VV^0\ \to\ B^0 \]
  is the universal family of all {\em smooth\/} divisors $V_b\subset W$ in the linear system $\vert 3H\vert$, so $\VV^0\subset\VV$ is a Zariski open.
  
  The family $\Ss$ is defined as
  \[ \Ss:= \VV/G\ \to\ B \]
  (where $G:=\ZZ_3^2$ and the $G$--action on $\VV$ is induced by the $G$-action on $W$).
\end{notation}

\begin{lemma}\label{smoothfam} Let $\VV\to B$ be as in notation \ref{not}(\rom1) or (\rom2). Then $\VV$ is a smooth quasi-projective variety.
\end{lemma}

\begin{proof} Let 
  \[ \bar{\VV}\ \subset \ \bar{B}\times W \]
  be the universal family of all (possibly very singular) divisors $V_b$ in the given linear system, where 
  \[\bar{B}=\PP H^0(W,\OO_W(3))^G\cong\PP^r\]
  is the closure of $B$. In both cases the complete linear system $\bar{B}$ is base point free (this is \cite[Section 3.1]{MLP} for type A, and \cite[Section 3.2]{MLP} for type B1), and so $\bar{\VV}$ is a $\PP^{r-1}$-bundle over $W$. As such, since $W$ is smooth, also $\bar{\VV}$ is smooth. It follows that the Zariski open $\VV\subset\bar{\VV}$ is smooth.
    \end{proof}

\section{An intermediate result}

\begin{proposition}\label{van} Let $\VV^0\to B^0$ be as in notation \ref{not} (\rom1) or (\rom2). Let $\VV^1\to B^1$ be the base change, where $B^1\subset B^0$ is a Zariski open. Then
  \[ B^2_{hom}(\VV^1\times_{B^1} \VV^1)_{\QQ}=0\ .\]
  \end{proposition}

\begin{proof} The smooth quasi--projective variety $\VV^0\times_{B^0} \VV^0$ is a Zariski open in the projective scheme $\bar{\VV}\times_{\bar{B}}\bar{\VV}$. Here, $\bar{B}\cong\PP^r$ is the closure of $B$, and $\bar{\VV}$ is the universal family of all (possibly very singular) divisors $V_b\subset W$ in the linear system $\vert 3H\vert$ (as in the proof of lemma \ref{smoothfam}).

\begin{lemma}\label{irred} The fibre product $\bar{\VV}\times_{\bar{B}}\bar{\VV}$ consists of an irreducible component $F$ of dimension $r+4$ (containing $\VV^0\times_{B^0} \VV^0$), plus perhaps some other components that lie over the diagonal $\Delta_W$ (and hence are of dimension $\le r+2$).
\end{lemma}

\begin{proof} Suppose $\bar{\VV}\times_{\bar{B}}\bar{\VV}$ is {\em not\/} irreducible, i.e. suppose one can write
  \[ \bar{\VV}\times_{\bar{B}}\bar{\VV} =  F\cup F^\prime\ ,\]
  where $F$ is the closure of the Zariski open $\VV^0\times_{B^0} \VV^0$, and $F^\prime$ is a union of other irreducible components.
  As $\VV^0\times_{B^0} \VV^0$ is irreducible, $F^\prime$ must lie over the locus $\Delta:=\bar{B}\setminus B^0$. Let
   \[  p\colon \bar{\VV}\times_{\bar{B}}\bar{\VV} \to W\times W \]
   be the projection, and suppose $p(F^\prime)$ is {\em not\/} contained in the diagonal $\Delta_W$.  
  For any point $x\in F^\prime$, the fibre $p^{-1}(p(x))$ is irreducible (sublemma \ref{general} below). 
  But if $x\in F^\prime$ is a point not lying over the diagonal $\Delta_W$, the fibre $p^{-1}(p(x))$ meets $F$ (sublemma \ref{general} below).    
  Thus, we see that this fibre is reducible, because we can write
  \[ p^{-1}(p(x)) =  (p\vert_{F})^{-1}(p(x))\cup    (p\vert_{F^\prime})^{-1}(p(x))\ , \]
 where both components are non--empty, distinct and a union of irreducibles. Contradiction; and so we must have $p(F^\prime)\subset \Delta_W$. This forces
 $\dim F^\prime\le 3+(r-1)=r+2$.
 
 \begin{sublemma}\label{general} For any point $y\in W\times W$, the fibre $p^{-1}(y)\subset \bar{\VV}\times_{\bar{B}}\bar{\VV}$ is isomorphic to $\PP^{s}$ (where $s\in\{r-2,r-1\}$). Moreover, for any point $y\in (W\times W)\setminus \Delta_W$, the fibre $p^{-1}(y)$ meets the Zariski open $ \VV^0\times_{B^0} \VV^0$.
  \end{sublemma} 
  
  To prove the first assertion of the sublemma, one remarks that (as the linear system $\bar{B}$ is base point free) a point $y=(w_1,w_2)\in W\times W$ imposes either one or two conditions on $\bar{B}$. For the ``moreover'' part, one
  needs to show that for any two points $w_1,w_2\in W$ with $w_1\not=w_2$, there exists a {\em smooth\/} divisor $V_b$ in the linear system $\bar{B}$ passing through $w_1$ and $w_2$. Since the linear system $\bar{B}$ is base point free, this follows from an easy blow--up argument, cf. \cite[Theorem 2.1]{DH}.
  
   \end{proof}

Let $F\subset \bar{\VV}\times_{\bar{B}} \bar{\VV}$ be the irreducible component as in lemma \ref{irred}. In order to prove proposition \ref{van}, it will suffice to prove the vanishing
  \begin{equation}\label{van2}  B_{r+2}^{hom}(F)_{\QQ}\stackrel{??}{=}0\ .\end{equation}
  (Indeed, using the Lefschetz (1,1) theorem on a resolution of singularities of the complementary divisor $F\setminus (\VV^1\times_{B^1} \VV^1)$, one finds that the natural map induced by restriction
   \[ B_{r+2}^{hom}(F)_{\QQ}\ \to\   B_{r+2}^{hom}(\VV^1\times_{B^1} \VV^1)_{\QQ})=  B^2_{hom}(\VV^1\times_{B^1} \VV^1)_{\QQ} \]
   is surjective.)
   
 Now, let us consider the morphism $p_F\colon F\to W\times W$ (which is the restriction of the morphism induced by projection $p\colon \bar{\VV}\times_{\bar{B}} \bar{\VV}\to W\times W$). We have seen (lemma \ref{irred}) that the inverse image of the diagonal $p^{-1}(\Delta_W)$ has dimension $\le r+2$. Hence, there are the following two possibilities:
 
 \vskip0.3cm
 \noindent
 {\underline{Case 1:}} $\dim (p_F)^{-1}(\Delta_W)=r+2$;
 
\vskip0.3cm
 \noindent
 {\underline{Case 2:}} $\dim (p_F)^{-1}(\Delta_W)<r+2$. 
 
 \vskip0.3cm
 Let us first treat case 1. In this case, lemma \ref{irred} implies that $F^\prime$ is empty, and so $ \bar{\VV}\times_{\bar{B}} \bar{\VV}=F$ is irreducible. 
   To prove the vanishing (\ref{van2}), we employ a stratification argument as in \cite{moit}. That is, there are natural morphisms
   \[  \begin{array}[c]{ccc}
         F= \bar{\VV}\times_{\bar{B}} \bar{\VV} & \xrightarrow{q} & \bar{B}\cong \PP^r   \\
          &&\\
          \downarrow{ p} &&\\
          &&\\
          W\times W
          \end{array}\]
          (where $W$ is the ambient threefold as in notation \ref{not}).
  Let us consider a diagram
   \[ \begin{array}[c]{ccccc}
      N^0 & \hookrightarrow &  \bar{\VV}\times_{\bar{B}} \bar{\VV} & \hookleftarrow & N_1\\
      &&&&\\
        \ \  \downarrow{p^0} &&\ \  \downarrow{p} &&\ \  \downarrow{p_1}\\
        &&&&\\
     M^0 & \hookrightarrow & W\times W & \hookleftarrow & M_1\\
     \end{array}\]
     Here, $M_1\subset W\times W$ is defined as the closed subvariety where the fibre dimension of $p$ is $> r-2$, and $M^0$ is the open complement $M^0:=(W\times W)\setminus M_1$. This shows that $p$ is a ``stratified projective bundle''. That is, letting $p^0$ and $p_1$ denote the restriction of $p$ to $N^0:=p^{-1}(M^0)$ resp. to $N_1:=p^{-1}(M_1)$, we have that $p^0$ is a $\PP^{r-2}$-fibration, and $p_1$ is a $\PP^{r-1}$-fibration. Note that $\dim M_1\le 4$ (since otherwise the subvariety $N_1=(p_1)^{-1}(M_1)$ would be of dimension $\ge r+4$, in contradiction with the irreducibility of $\bar{\VV}\times_{\bar{B}} \bar{\VV} $).
      
 We are now ready to prove the vanishing (\ref{van2}). First, we note that $W\times W$ has trivial Chow groups, which implies that $W\times W$ has the strong property in any degree (indeed, rational and homological equivalence coincide on $W\times W$, and $W\times W$ has no odd--degree cohomology). 
 Lemma \ref{cut} thus implies that $M^0$ has the strong property in degree $4$. Lemma \ref{projbun} then implies that $N^0$ has the strong property in degree $4+(r-2)=r+2$.
 As $\dim M_1\le 4$, $M_1$ has the weak property in degree $3$ (for dimension reasons). Lemma \ref{projbun} implies that $N_1$ has the weak property in degree $3+(r-1)=r+2$. Applying lemma \ref{paste}, we find that $ \bar{\VV}\times_{\bar{B}} \bar{\VV}$ has the weak property in degree $r+2$, which by definition means that
   \[    B^{hom}_{r+2}( \bar{\VV}\times_{\bar{B}} \bar{\VV})_{\QQ}= \ker\Bigl( L_{r+2} H_{2r+4}( \bar{\VV}\times_{\bar{B}} \bar{\VV})\to H_{2r+4}( \bar{\VV}\times_{\bar{B}} \bar{\VV})\Bigr) =0\ . \]
   This proves the desired vanishing (\ref{van2}) in case 1.
   
  It remains to treat case 2. When we are in case 2, we may simply disregard what happens over the diagonal $\Delta_W$. More precisely,  let us define 
   \[ F^0:= F\setminus (p_F)^{-1}(\Delta_W)\ .\]   
  Since the complement $ (p_F)^{-1}(\Delta_W)$ has dimension $\le r+1$, restriction induces an isomorphism
   \[ B_{r+2}(F)\ \xrightarrow{\cong}\ B_{r+2}(F^0)\ .\]
   As such, in case 2 the vanishing (\ref{van2}) is equivalent to the vanishing
   \begin{equation}\label{van3} B_{r+2}^{hom}(F^0)_{\QQ}\stackrel{??}{=}0\ .\end{equation}
  To prove the vanishing (\ref{van3}), we play the same game (for $F^0$) that we played in case 1 (for $F$). That is, we consider a diagram of fibre squares
    \[   \begin{array}[c]{ccccc}
      F^1 & \hookrightarrow &   F^0 & \hookleftarrow & F_1\\
      &&&&\\
        \ \  \downarrow{p^1} &&\ \  \downarrow{p^0} &&\ \  \downarrow{p_1}\\
        &&&&\\
     M^0\setminus (\Delta_W\cap M^0) & \hookrightarrow & (W\times W)\setminus \Delta_W & \hookleftarrow & M_1\setminus (\Delta_W\cap M_1)  \\
     \end{array}\]
   Here the arrows $p^0, p^1,p_1$ are obtained by restricting the morphism $p_F\colon F\to W\times W$ to $F^0$, resp. to $F^1$ resp. $F_1$. We still have (just as in case 1) that $p^1$ is a $\PP^{r-2}$-fibration, and $p_1$ is a $\PP^{r-1}$-fibration. 
   
  As we have seen, $W\times W$ has the strong property in any degree, and the diagonal $\Delta_W$ is of dimension $3$. Lemma \ref{cut} then implies that $(W\times W)\setminus \Delta_W$ has the strong property (in degree $3$ and hence a fortiori) in degree $4$.  Applying lemma \ref{cut} once more, we find that also $M^1$ has the strong property in degree $4$. Lemma \ref{projbun} then implies that $F^1$ has the strong property in degree $4+(r-2)=r+2$. The 
  complement $M_1\setminus (\Delta_W\cap M_1)$, being of dimension $\le 4$, has the weak property in degree $3$. Lemma \ref{projbun} now implies that $F_1$ has the weak property in degree $3+(r-1)=r+2$. Applying lemma \ref{paste}, we find that $ F^0 $ has the weak property in degree $r+2$, which is exactly the desired vanishing (\ref{van3}).
   
Proposition \ref{van} is now proven.  
   \end{proof}

\begin{remark} The main place where the present note diverges from the argument of \cite{V8} is in the proof of proposition \ref{van}. Indeed, in \cite{V8} the vanishing of the codimension $2$ Griffiths group of a certain fibre product is obtained by establishing that this fibre product has a smooth projective compactification that is rationally connected \cite[Proposition 2.2]{V8}. Since this seems unfeasible in the context of the present note, we have replaced this by the stratification argument of proposition \ref{van}.

\end{remark}

\begin{remark}\label{more} For possible later use, we remark that the argument proving proposition \ref{van} actually applies to the following more general setting: $W$ is a smooth projective threefold with trivial Chow groups, $\bar{B}$ is a base point free linear system on $W$.
\end{remark}

\section{The main result}

\begin{theorem}\label{main} 
Let $S$ be a numerical Campedelli surface with $\ord(G)=9$. Then
  \[ A^2(S)=\ZZ\ .\]
  \end{theorem}
  
  \begin{proof} First of all, we note that $q(S)=0$ and so, thanks to Rojtman's result \cite{Roj}, it will suffice to prove that $A^2(S)_{\QQ}=\QQ$. 
  
 Let $\bar{S}$ be the canonical model of $S$. Let us assume that $\bar{S}$ is of type A or type B1, in the language of theorem \ref{const}. (We can treat these two cases simultaneously; the remaining case where $S$ is of type B2 will be dealt with separately.) This means that $\bar{S}=V/G$, where $V$ is a surface with rational double points and $G$ a finite abelian group. The surface $V$ is an Alexander scheme (rational double points can be resolved by so-called A-D-E curves), and so $\bar{S}$ is an Alexander scheme (lemma \ref{alexquot}).
 Since $A_0()_{\QQ}$ is a birational invariant of surfaces with rational singularities (proposition \ref{birat}), it will suffice to prove that the canonical model $\bar{S}$ has
    \[ A_0(\bar{S})_{\QQ}\stackrel{??}{=}\QQ\ .\]
  In view of lemma \ref{alexquot}, it thus suffices to prove that
   \begin{equation}\label{goal} A_0(V)^G_{\QQ} \stackrel{??}{=}\QQ\ .\end{equation}
   
   We have seen (theorem \ref{const}) that the surface $V$ is a fibre $V_b$ of the universal family $\VV\to B$ as in notation \ref{not}(\rom1) or (\rom2). To prove that every $V_b$ satisfies (\ref{goal}), it suffices (in view of proposition \ref{gen}) to prove that the very general $V_b$ satisfies (\ref{goal}). This will be convenient: in the course of the argument, we are at liberty to replace the base $B$ by a smaller (non-empty) Zariski open. In particular, from now on we can and will suppose that the base of our family $B^1$ is contained in $B^0$, i.e. every fibre $V_b$ is smooth.
   
   The fact that the canonical model $\bar{S}=V_b/G$ has $p_g(\bar{S})=0$ means that
   \[ H^{2,0}(V_b)^G=0\ \ \ \forall\ b\in B^1\ .\]
   Using the Lefschetz $(1,1)$ theorem, this means that the cycle class map induces an isomorphism
   \begin{equation}\label{nsiso} NS(V_b)^G\ \xrightarrow{\cong}\ H^2(V_b)^G\ \ \ \forall\ b\in B^1\ .\end{equation}   
   Let us consider the relative correspondence
 \[ \Gamma:=  \Delta_G^- \ \ \ \in\ A^2(\VV^1\times_{B^1} \VV^1)_{\QQ}\ ,\]
   where $\Delta_G^-$ is (the restriction to the smaller base $B^1$ of) the relative correspondence of the proof of proposition \ref{gen}.
   For any $b\in B$, the fibrewise restriction
    \[ \Gamma_b:=\Gamma\vert_{V_b\times V_b}\ \ \ \in\ A^2(V_b\times V_b)_{\QQ} \]
    acts on cohomology as a projector onto $H^2(V_b)^G$, and acts on $A^2(V_b)_{\QQ}$ as a projector onto $A^2_{hom}(V_b)_{\QQ}^G$.
    Because $\Gamma_b$ acts on cohomology as projector onto $H^2(V_b)^G$, we have
    \[ \Gamma_b\ \ \ \in\     H^2(V_b)^G \otimes H^2(V_b)^G \ \ \   \subset\ H^4(V_b\times V_b)\ \ \  \forall\ b\in B^1\ .\]
    Because of the isomorphism (\ref{nsiso}), this implies the following: for any $b\in B^1$, there exists a divisor $D_b\subset V_b$, and a cycle $\gamma_b$ supported on $D_b\times D_b$, with the property that
     \[ \Gamma_b=\gamma_b\ \ \ \hbox{in}\ H^4(V_b\times V_b)\ .\]
     Thanks to Voisin's Hilbert schemes argument \cite[Proposition 3.7]{V0} (which is the same as \cite[Proposition 4.25]{Vo}), these data can be ``spread out'' over the family. That is, there exists a divisor $\DD\subset\VV^1$, and a cycle $\gamma\in A^2(\VV^1\times_{B^1} \VV^1)_{\QQ}$ with support on $\DD\times_{B^1} \DD$, such that
     \[ (\Gamma -\gamma)_b:=  (\Gamma -\gamma)\vert_{V_b\times V_b}  =0\ \ \ \hbox{in}\ H^4(V_b\times V_b)\ \ \ \forall b\in B^1\ .\]     
  
  That is, the relative cycle $\Gamma-\gamma\in A^2(\VV^1\times_{B^1} \VV^1)_{\QQ}$ is fibrewise homologically trivial. The next step is to make this relative cycle {\em globally\/} homologically trivial. This can be done thanks to Voisin's Leray spectral sequence argument \cite[Lemmas 3.11 and 3.12]{V0} (cf. also \cite[Lemma 1.2]{V8}). 
  The outcome of this is as follows: there exists a cycle $\delta\in A^2(W\times W)_{\QQ}$ (where $W$ is either $\PP^1\times \PP^1\times\PP^1$ or a flag variety in $\PP^2\times\PP^2$, depending on whether we are in case (\rom1) or in case (\rom2) of theorem \ref{const}), with the property that after shrinking $B^1$ we have
  \[ \Gamma -\gamma - (B^1\times\delta)\vert_{\VV^1\times_{B^1} \VV^1}\ \ \ \in\ A^2_{hom}(\VV^1\times_{B^1} \VV^1)_{\QQ}\ .\]
  In view of proposition \ref{van}, this is the same as saying that
  \[ \Gamma -\gamma - (B^1\times\delta)\vert_{\VV^1\times_{B^1} \VV^1}\ \ \ \in\ A^2_{alg}(\VV^1\times_{B^1} \VV^1)_{\QQ}\ .\]  
  In particular, restricting to any fibre we obtain that
  \[ (\Gamma-\gamma)_b - \delta_b\ \ \ \in\ A^2_{alg}(V_b\times V_b)_{\QQ}\ \ \ \forall\ b\in B^1\ .\]
  But then, in view of the weak nilpotence result (proposition \ref{weaknilp}), this means that for any $b\in B^1$ there exists $N_b$ such that
  \begin{equation}\label{this} \Bigl( (\Gamma-\gamma)_b - \delta_b\Bigr)^{\circ N_b}=0\ \ \ \hbox{in}\ A^2_{}(V_b\times V_b)_{\QQ}\  .\end{equation}
  Since $A^\ast_{hom}(W)=0$, the restriction $\delta_b$ is a {\em degenerate correspondence\/} (meaning that it is supported on $D_b\times V_b\cup V_b\times D_b$, for some divisor $D_b\subset V_b$). Up to some further shrinking of $B^1$, we may assume that the restriction $\gamma_b$ is also a degenerate correspondence.
  Developing the expression (\ref{this}) for any given $b\in B^1$, we obtain a rational equivalence
  \begin{equation}\label{rateq}  (\Gamma_b)^{\circ N_b} -\gamma^\prime_b =0\ \ \ \hbox{in}\ A^2_{}(V_b\times V_b)_{\QQ}\   ,\end{equation}
  where $\gamma^\prime_b$ is a sum of compositions of correspondences containing either a copy of $\gamma_b$ or a copy of $\delta_b$.
  But degenerate correspondences form a two-sided ideal in the ring of correspondences \cite[Example 16.1.2(b)]{F}, and so $\gamma^\prime_b$ is degenerate. On the other hand, it is readily checked that  $\Gamma_b=\Delta^-_G\vert_{V_b\times V_b}$ is idempotent. Equality (\ref{rateq}) now simplifies to
  \begin{equation}\label{rat2} \Gamma_b -\gamma^\prime_b =0\ \ \ \hbox{in}\ A^2_{}(V_b\times V_b)_{\QQ}\   ,\end{equation}
  where $\gamma^\prime_b$ is degenerate. Clearly, degenerate correspondences act trivially on $A^2_{AJ}(V_b)_{\QQ}$, and so (\ref{rat2}) implies that
   \[ (\Gamma_b)_\ast A^2_{AJ}(V_b)_{\QQ}=0\ \ \ \ \forall b\in B^1.\]
   On the other hand, by construction $\Gamma_b$ is a projector on $A^2_{AJ}(V_b)_{\QQ}^G$, and so it follows that
   \[   A^2_{AJ}(S_b)_{\QQ}=A^2_{AJ}(V_b)_{\QQ}^G=0\ \ \ \forall\ b\in B^1\ .\]
   Since $q(S_b)=0$, this implies (\ref{goal}) for Campedelli surfaces of type A1 or B1.
   
   Finally, it remains to treat the case where the canonical model $\bar{S}$ is of type B2 (i.e., case (\rom3) of theorem \ref{const}). In that case (cf. theorem \ref{const}(\rom3)), we can find a family $\Ss\to B$ such that $\bar{S}$ is isomorphic to the fibre $S_0$, and a general fibre $S_b, b\not=0$ is a quotient variety of type B1.  The very general fibre will have $A^2(S_b)_\QQ=\QQ$ because it is a surface of type B1. In view of proposition \ref{gen} (with $G=\{\ide\}$), we then also have $A_0(S_0)_{\QQ}=\QQ$. (NB: rather than invoking proposition \ref{gen} at this point, we could alternatively apply \cite[Theorem 3.1]{KSZ} to conclude that $A_0(S_0)_{\QQ}=\QQ$.)  
   
   But the fibre ${S}_0$ is isomorphic to $\bar{S}$ and hence birational to $S$, and so we find (applying proposition \ref{birat}) that
   $A_0(S)_{\QQ}=\QQ$ in this case as well. 
   
   The theorem is now proven. 
  \end{proof}

  \section{A corollary} 
  
  \begin{corollary}\label{cor} Let $S$ be a numerical Campedelli surface with $G=\ZZ_3^2$ as in \cite[Section 5 Example 2]{CMLP}, and let $T=S/\sigma_1$ be as in loc. cit., so $T$ is a numerical Godeaux surface with $\pi_1(T)=\ZZ_3$. Then
  \[ A^2(T)=\ZZ\ .\]
   \end{corollary}
   
   \begin{proof}
   This is immediate from theorem \ref{main}.
   \end{proof}

\vskip1cm
\begin{nonumberingt} Thanks to Claudio Fontanari for inviting me to Trento, where this note was written. Thanks to Roberto Pignatelli for some inspiring conversations.
Grazie mille to Yoyo, Kai and Len for joining me in sunny Trento.
\end{nonumberingt}

\end{document}